\documentclass[12pt]{amsart}
\usepackage[utf8]{inputenc}
\usepackage{amsmath,amsfonts,amsthm, amssymb, xcolor}
\usepackage{mathtools}
\usepackage{graphicx}
\graphicspath{ {./images/} }
\usepackage{tikz}
\usetikzlibrary{arrows}
\usepackage{float}
\usepackage{caption}
\usepackage{subcaption}
\usepackage{paralist}
\usepackage[pagebackref]{hyperref}
\usepackage[nameinlink]{cleveref}

\usepackage{algorithm}
\usepackage[noend]{algpseudocode}

\newtheorem{theorem}{Theorem}[section]

\theoremstyle{definition}

\newtheorem{definition}{Definition}[section]

\newtheorem{remark}{Remark}[section]

\newtheorem{example}{Example}[section]

    \newcommand{\RR}    {\Bbb{R}}

    \newcommand{\NN}    {\Bbb{N}}
        \newcommand{\ep}    {\epsilon}

\hypersetup{
    colorlinks=true,
    linkcolor=blue,
    filecolor=magenta,
    urlcolor=cyan,
}

\usepackage[verbose=true,letterpaper]{geometry}
\AtBeginDocument{
  \newgeometry{
    textheight=9in,
    textwidth=6.5in,
    top=1in,
    headheight=12pt,
    headsep=25pt,
    footskip=30pt
  }
}

\title{On merge trees with a given homological sequence}
\author{Nicholas A. Scoville and Dylan Wen}

\address[Nicholas A. Scoville]{Department of Mathematics and Computer Science, Ursinus College, Collegeville PA 19426}
\email{nscoville@ursinus.edu}

\address[Dylan Wen]{Methacton High School, Eagleville, PA}
\email{dylanmwen@gmail.com}

\date{\today}

\begin{document}
\maketitle
\begin{abstract}
In this paper, we study the induced homological sequence and the induced merge tree of a discrete Morse function on a tree.  A discrete Morse function on a tree gives rise to a sequence of Betti numbers that keep track of the number of components at each critical value.  A discrete Morse function on a tree also gives rise to an induced merge tree which keeps track of component birth, death, and merging information. These topological indicators are similar but neither one contains the information of the other.  We show that given a merge tree and a homological sequence along with some mild conditions on their relationship, there is a discrete Morse function on a tree that induces both the given merge tree and the given homological sequence.

\end{abstract}
\tableofcontents

\section{Introduction}

The goal of this paper is to study the relationship between merge trees and the homological sequence induced by a discrete Morse function.  An induced homological sequence was originally defined by Nicolaescu \cite{N-08},  and it was studied in the discrete case by R. Ayala and several collaborators \cite{A-F-F-V-09, Ayala-11, Ayala-10a, Ayala-10b} as well as \cite{AGORS} and \cite{RandScoville}.  The homological sequence of a discrete Morse function associates the (multi)-sequence of the Betti numbers of each level subcomplex at a critical simplex.

In addition to keeping track of Betti numbers at each critical simplex, one may also keep a detailed account of component information by studying the merge tree associated to a discrete Morse function. Merge trees were first studied in the context of discrete Morse functions in \cite{JohnsonScoville2022}. Since any discrete Morse function gives rise to a natural filtration, a merge tree can be associated to any discrete Morse function on a fixed graph. The authors studied the realization problem in some special cases, and Br\"uggemann in \cite{Bruggemann2023} proved that any given merge tree can be realized as the induced merge tree of a discrete Morse function on a path.  See section \ref{sec:realization} for details.

We combine these two trackers of topological information by asking: given a merge tree and homological sequence, is there a discrete Morse function on a graph whose induced merge tree is the given merge tree and whose induced homological sequence is the given homological sequence?  With mild restrictions on the homological sequence, we answer in the affirmative in Theorem \ref{thm: main}. The structure of the paper is given as follows.  In section \ref{sec:background}, we give the necessary background in graph theory and discrete Morse theory.  Section \ref{sec:merge} is devoted to defining merge trees induced from a discrete Morse function on a tree. We also review Br\"uggemann's construction \cite{Bruggemann2023} and present it in algorithmic form (Algorithm \ref{Jalg}). Finally, our main theorem is proved in section \ref{sec:merge and homology} along with an example illustrating the construction.

\section{Background}\label{sec:background}

\subsection{Graphs and trees}
Let $G=(V(G),E(G))$ be a finite, loopless graph without multi-edges (i.e. a $1$-dimensional abstract simplicial complex).   We call an edge or a vertex of $G$ a \textbf{simplex}. If $e=uv$ is an edge, we say that the edge $e$ is \textbf{incident} with vertex $v$ and that $u$ and $v$ are \textbf{adjacent}. We use $|V(G)|$ to denote the number of vertices of $G$ and $|E(G)|$ to denote the number of edges of $G$.\\

\noindent A connected graph without any cycles is called a \textbf{tree}.  A disconnected graph $F$ such that each component of $F$ is a tree is called a \textbf{forest}.  For any vertex $v\in F$, we let $F[v]$ denote the connected component of $F$ containing $v$. It immediately follows that if $F$ is a forest with two distinct vertices $u,v\in F$, then there is a path between two vertices $u$ and $v$ if and only if $F[u]=F[v]$.

\subsection{Discrete Morse theory}
Our references for the basics of discrete Morse theory are \cite{Forman-2002,KnudsonBook, Koz20, DMTSco}.  There are several different ways of viewing a discrete Morse function.  For our purposes, we make the following definition:

\begin{definition}\label{discreteMF}  Let $G$ be a graph.  A function $f\colon G \to \RR$ is called  \textbf{weakly increasing} if $f(v)\leq f(e)$ whenever $v \subseteq e$.  A \textbf{discrete Morse function} $f\colon G \to \RR$ is a weakly increasing function which is at most 2--1 and satisfies the property that if $f(\sigma)=f(\tau)$, then either $\sigma\subseteq \tau$ or $\tau\subseteq \sigma$. Any simplex $\sigma$ on which $f$ is 1--1 is called \textbf{critical} and the value $f(\sigma)$ is a \textbf{critical value} of $f$.
\end{definition}

\begin{example}\label{ex: dmf} To illustrate definition \ref{discreteMF}, define the discrete Morse function $f$ on the graph below as follows:
$$
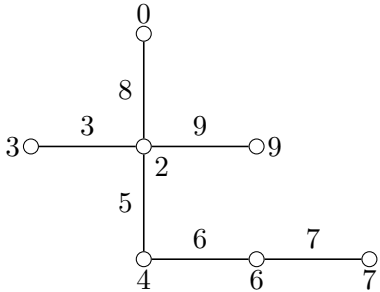
\begin{tikzpicture}[scale=.75]

\node[inner sep=2pt, circle] (7) at (2,2) [draw] {};
\node[inner sep=2pt, circle] (9) at (4,2) [draw] {};
\node[inner sep=2pt, circle] (6) at (0,2) [draw] {};
\node[inner sep=2pt, circle] (1) at (0,4) [draw] {};
\node[inner sep=2pt, circle] (0) at (0,6) [draw] {};
\node[inner sep=2pt, circle] (2) at (-2,4) [draw] {};
\node[inner sep=2pt, circle] (4) at (2,4) [draw] {};

\path[style=semithick] (0) edge node[anchor=east]{\small{$8$}}(1);
\path[style=semithick] (1) edge node[anchor=east]{\small{$5$}}(6);
\path[style=semithick] (2) edge node[anchor=south]{\small{$3$}}(1);
\path[style=semithick] (1) edge node[anchor= south]{\small{$9$}}(4);
\path[style=semithick] (6) edge node[anchor=south]{\small{$6$}}(7);
\path[style=semithick] (7) edge  node[anchor= south]{\small{$7$}}(9);

\node[anchor = south ]  at (0) {\small{$0$}};
\node[anchor = north west]  at (1) {\small{$2$}};
\node[anchor =  east]  at (2) {\small{$3$}};
\node[anchor = west]  at (4) {\small{$9$}};
\node[anchor = north]  at (6) {\small{$4$}};
\node[anchor = north]  at (7) {\small{$6$}};
\node[anchor = north]  at (9) {\small{$7$}};
\end{tikzpicture}
$$

\noindent The critical vertices are $f^{-1}(0), f^{-1}(2),$ and $f^{-1}(4)$ while the critical edges are $f^{-1}(5)$ and $f^{-1}(8)$.

\end{example}

\begin{definition}\label{level} Let $G$ be a graph and $f\colon G\to \RR$ a discrete Morse function.  Given $a \in \mathbb R$ the \textbf{level subcomplex} $G_a$ is defined to be the induced subgraph of $G$ consisting of all simplices $\sigma$ with $f(\sigma) \leq a$.  For each critical value $c_0< \ldots < c_{m-1}$ of $f$, we consider the induced sequence of level subcomplexes $\{v\}=G_{c_0} \subset G_{c_1}\subset \ldots \subset G_{c_{m-1}}$.  In the sequel, we will use the notation $G_{c_i-\ep}$ to denote the level subcomplex immediately preceding $G_{c_i}$; that is, $\ep$ is chosen so that $f(\sigma)<c_i-\ep<c_i$ for every $\sigma\in G$ such that $f(\sigma)<c_i$.
\end{definition}

\begin{definition} Let $f$ be a discrete Morse function with $m$ critical values on a tree $T$.  The \textbf{homological sequence of $f$} is given by
$$B^f_0\colon \{0, 1, \ldots, m-1\}\to
\mathbb{N}\cup \{0\}$$
 defined by $B^f_0(i):=b_0(T(c_i))$ for all $0\leq i\leq m-1$.  If $g\colon T \to \RR$ is another discrete Morse functions $m$ critical values, we say $f$ is \textbf{less than or equal to $g$}, denoted $f\leq g$, if $B_0^f(i)\leq B_0^g(i)$ for every $0\leq i\leq m-1$.
\end{definition}

\begin{example} Continuing with Example \ref{ex: dmf}, the critical values are $0<2<4<5<8$ which induce level subcomplexes $G_0\subseteq G_2\subseteq G_4\subseteq G_5\subseteq G_8$. Considering the sequence of these subgraphs and their Betti numbers, we obtain the homological sequence of $f$ which is shown below:

$$
\begin{array}{c|cccccccc}
i& 0 & 1 & 2 & 3 & 4  \\
\hline
B_0(i)& 1 & 2 & 3 & 2 & 1  \\
\end{array}
$$

\end{example}

\section{Merge trees}\label{sec:merge}

In this section we introduce merge trees, our main object of study.

\subsection{Basics of merge trees}

\begin{definition}\label{defn merge tree} A \textbf{binary tree} is a rooted tree where each node has at most two children, and each child is designated as its \textbf{left} (L) or \textbf{right} (R) child, called the \textbf{chirality} of the node.  A binary tree is \textbf{full} if every node has $0$ or $2$ children. A \textbf{(chiral) merge tree} is a full binary tree.
\end{definition}

We let $c(v)$ denote the chirality of $v$; that is, $c(v)=\mathrm{L}$ or $c(v)=\mathrm{R}$.

\begin{remark}
Although graphs and merge trees are different objects, they both consist of vertices and edges.  To help distinguish them, we reserve the term ``node" for merge trees and ``vertex" for graphs.
\end{remark}

The nodes of a merge tree may be partitioned into nodes of degree 1, called \textbf{leafs}, and nodes of degree 2 or 3, called \textbf{inner nodes}.  The number of leafs in a merge tree $T$ is denoted $\ell(T)$ while the number of inner nodes is denoted $i(T)$.  Note that the \textbf{root}, or unique vertex of degree 2 (denoted $r_M=r$), of a merge tree is considered an inner node.  It is a well known fact in graph theory that $i(T)=\ell(T)+1$.

\begin{example}\label{ex: merge tree running example}
An example of a merge tree $M$ is given below.
$$
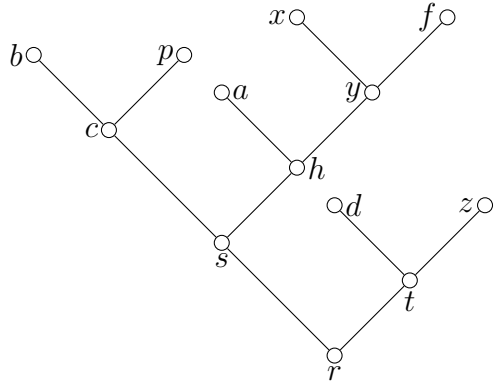
\begin{tikzpicture}

\node[inner sep=2pt, circle] (c) at (-.5,2.5) [draw] {};
\node[inner sep=2pt, circle] (b) at (-1.5,3.5) [draw] {};
\node[inner sep=2pt, circle] (p) at (.5,3.5) [draw] {};
\node[inner sep=2pt, circle] (h) at (2,2) [draw] {};
\node[inner sep=2pt, circle] (y) at (3,3) [draw] {};
\node[inner sep=2pt, circle] (f) at (4,4) [draw] {};
\node[inner sep=2pt, circle] (x) at (2,4) [draw] {};
\node[inner sep=2pt, circle] (a) at (1,3) [draw] {};
\node[inner sep=2pt, circle] (t) at (3.5,.5) [draw] {};
\node[inner sep=2pt, circle] (d) at (2.5,1.5) [draw] {};
\node[inner sep=2pt, circle] (s) at (1,1) [draw] {};
\node[inner sep=2pt, circle] (r) at (2.5,-.5) [draw] {};
\node[inner sep=2pt, circle] (z) at (4.5,1.5) [draw] {};

\draw[-]  (c)--(s) node[midway, below] {};
\draw[-]  (s)--(h) node[midway, below] {};
\draw[-]  (s)--(r) node[midway, below] {};
\draw[-]  (t)--(r) node[midway, below] {};
\draw[-]  (t)--(z) node[midway, below] {};
\draw[-]  (t)--(d) node[midway, below] {};
\draw[-]  (h)--(a) node[midway, below] {};
\draw[-]  (h)--(y) node[midway, below] {};
\draw[-]  (y)--(x) node[midway, below] {};
\draw[-]  (y)--(f) node[midway, below] {};
\draw[-]  (c)--(b) node[midway, below] {};
\draw[-]  (c)--(p) node[midway, below] {};

\node[anchor = east]  at (b) {{$b$}};
\node[anchor = east]  at (c) {{$c$}};
\node[anchor = west]  at (h) {{$h$}};
\node[anchor = north]  at (s) {{$s$}};
\node[anchor = north]  at (r) {{$r$}};
\node[anchor = north]  at (t) {{$t$}};
\node[anchor = east]  at (p) {{$p$}};
\node[anchor = east]  at (x) {{$x$}};
\node[anchor = east]  at (y) {{$y$}};
\node[anchor = east]  at (f) {{$f$}};
\node[anchor = west]  at (a) {{$a$}};
\node[anchor = west]  at (d) {{$d$}};
\node[anchor = east]  at (z) {{$z$}};

\end{tikzpicture}
$$
We label each of the nodes with arbitrary names, as we will put a special ordering on the nodes below.  We see that the number of inner nodes is $i(M)=6$ and number of leaves is $\ell(M)=7.$
\end{example}

Given any merge tree $M$, it is not difficult to see that every node $a\in M$ is uniquely determined by the shortest path from the root node $r_M=r$ to $a$.  The \textbf{depth} of $a$ is the length of its shortest path from the root node.  We then associate to node $a$ the path word of $a$ given by the sequence of left/right moves from the root vertex.  Formally, if $r=a_0, a_1, \ldots, a_k=a$ is the sequence of vertices in the unique path from $r$ to $a$, the \textbf{path word of $a$}, denoted $P(a)$, is the sequence $c(a_0)c(a_1)\ldots c(a_{k-1})$.  We will find it convenient to extend the length the word path of a node by considering the empty letter, denoted $\underline{\hspace{8pt}}$.

We put an ordering $\leq$ on the nodes of any merge tree $M$ as follows.  Let $a,b$ be nodes in $M$ with corresponding path words $a_0a_1\ldots a_n$ and $b_0b_1 \ldots b_m$, respectively. Since both paths begin at the root node $r$, $a_0=b_0=c(r)=L$, so that there is a maximal integer $k\leq 1$ such that $a_i=b_i$ for all $i\leq k$.  If $a_k=b_k=L$ ($a_k=b_k=R$), define $a\leq b$ if and only if one of the following holds:
\begin{enumerate}
    \item[(a)] $a_{k+1}=L$ and $b_{k+1}=R$ ($a_{k+1}=R$ and $b_{k+1}=L$)
    \item[(b)] $b_{k+1}=\underline{\hspace{8pt}}$
    \item[(c)] $a=b$.
\end{enumerate}

This ordering on the nodes is called the \textbf{sublevel-connected Morse order} \cite[Definition 4.5]{Bruggemann2023}. We will construct an interleaving of this ordering for the proof of Theorem \ref{thm: main} below.

\begin{example} The above definition essentially boils down to the idea that the more the path word fluctuates, the larger in the ordering the corresponding node will be.  A few examples will illustrate. Consider the following nodes and their corresponding path words from Example \ref{ex: merge tree running example}:

\begin{eqnarray*}
P(b)&=& LLLL\\
P(a)&=& LLRL\\
P(d)&=& LRL\\
P(r)&=& L.\\
\end{eqnarray*}

Then $b$ is the smallest node in the ordering while $r$ is the largest node.  We furthermore see that since $a$ and $d$ agree at the first value but disagree on the second value, the value that changes corresponds to the larger node.  Hence $a\leq d$.
\end{example}

It can be shown that the order $\leq$ defined above is a total order on the nodes of $M$ in the sense that $v_0\leq v_1\leq \cdots \leq v_{i(M)+\ell(M)-1}$.  This means that we can associate a unique number or labeling $\lambda\colon V(T)\to \{0,1, \ldots, i(M)+\ell(M)-1\}$ to the nodes in $M$ based on their place in the total order.

\begin{example}
We used the order to put the nodes of $M$ in Example \ref{ex: merge tree running example} into a total order.  Now we label each node of $M$ by its value $\lambda(v)$:

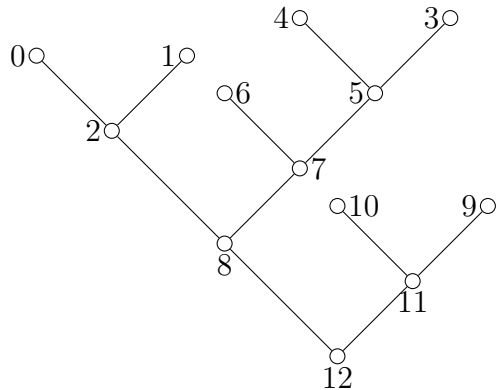
$$
\begin{tikzpicture}

\node[inner sep=2pt, circle] (c) at (-.5,2.5) [draw] {};
\node[inner sep=2pt, circle] (b) at (-1.5,3.5) [draw] {};
\node[inner sep=2pt, circle] (p) at (.5,3.5) [draw] {};
\node[inner sep=2pt, circle] (h) at (2,2) [draw] {};
\node[inner sep=2pt, circle] (y) at (3,3) [draw] {};
\node[inner sep=2pt, circle] (f) at (4,4) [draw] {};
\node[inner sep=2pt, circle] (x) at (2,4) [draw] {};
\node[inner sep=2pt, circle] (a) at (1,3) [draw] {};
\node[inner sep=2pt, circle] (t) at (3.5,.5) [draw] {};
\node[inner sep=2pt, circle] (d) at (2.5,1.5) [draw] {};
\node[inner sep=2pt, circle] (s) at (1,1) [draw] {};
\node[inner sep=2pt, circle] (r) at (2.5,-.5) [draw] {};
\node[inner sep=2pt, circle] (z) at (4.5,1.5) [draw] {};

\draw[-]  (c)--(s) node[midway, below] {};
\draw[-]  (s)--(h) node[midway, below] {};
\draw[-]  (s)--(r) node[midway, below] {};
\draw[-]  (t)--(r) node[midway, below] {};
\draw[-]  (t)--(z) node[midway, below] {};
\draw[-]  (t)--(d) node[midway, below] {};
\draw[-]  (h)--(a) node[midway, below] {};
\draw[-]  (h)--(y) node[midway, below] {};
\draw[-]  (y)--(x) node[midway, below] {};
\draw[-]  (y)--(f) node[midway, below] {};
\draw[-]  (c)--(b) node[midway, below] {};
\draw[-]  (c)--(p) node[midway, below] {};

\node[anchor = east]  at (b) {{$0$}};
\node[anchor = east]  at (c) {{$2$}};
\node[anchor = west]  at (h) {{$7$}};
\node[anchor = north]  at (s) {{$8$}};
\node[anchor = north]  at (r) {{$12$}};
\node[anchor = north]  at (t) {{$11$}};
\node[anchor = east]  at (p) {{$1$}};
\node[anchor = east]  at (x) {{$4$}};
\node[anchor = east]  at (y) {{$5$}};
\node[anchor = east]  at (f) {{$3$}};
\node[anchor = west]  at (a) {{$6$}};
\node[anchor = west]  at (d) {{$10$}};
\node[anchor = east]  at (z) {{$9$}};

\end{tikzpicture}
$$
Labeling a merge tree by the values $\lambda(v)$ will be a convenient way to not only read off the ordering of the vertices, but also follow the proof of Theorem \ref{thm: main}.
\end{example}

\subsection{The realization problem}\label{sec:realization}

In this section, we present in an algorithmic form the construction given in \cite{Bruggemann2023} which associates to any merge tree $M$ a path $P$ and a discrete Morse function $f\colon P\to \RR$ such that $M_f=M$.
  This will be used as the basis for our result in Theorem \ref{thm: main} below.

\begin{algorithm}
  \caption{Merge tree algorithm}\label{Jalg}
\verb"Input:" A chiral merge tree $M$ on $n=2k+1>0$ nodes. \\
\verb"Output:" A discrete Morse function $f\colon P \to \mathbb{N}$
whose \\induced merge tree $M_f$ is equal to $M$.

\begin{enumerate}
\item[a)] Order the  nodes $s_0\leq s_1\leq \cdots \leq s_{n-1}$ of $M$ using the  sublevel-connected Morse order $\leq.$
\item[b)] Initialize $j=0$, and $P=\{v_0\}$. Set $j=1$.
\item[c)] If $j=n$, end the algorithm. Otherwise
    \begin{enumerate}
     \item   If $s_j$ is a leaf, define $P=P\cup\{v_j\}$ and define $f(v_j)=j$. Set $j=j+1$ and return to c).
    \item    If $s_j$ is an inner node with children $x$ and $y$, define $P=P\cup\{e_j\}$ where $e_j$ is an edge between a vertex of degree 1 in $P[x]$ and a vertex of degree 1 in $P[y]$. Define $f(e_j)=j$.  Set $j=j+1$ and return to c).
     \end{enumerate}
\end{enumerate}
\end{algorithm}

That this algorithm is correct is proved in \cite[Theorem 5.5]{Bruggemann2023} (though not in algorithmic form). For a given merge tree $M$, we will denote the homological sequence induced by Algorithm \ref{Jalg} by $J_0^M=J_0$.

\section{Merge trees and homological sequences}\label{sec:merge and homology}

This section is devoted to our main result of constructing a discrete Morse function on a graph that induces both a given homological sequence and a given merge tree.  As we will see, not all combinations of homological sequences and merge trees are possible.  We begin with the following compatibility criteria.

\begin{definition} We say that a merge tree $M$ and homological sequence $B_0$ are \textbf{consistent} if the number of nodes of $M$ is equal to the number of elements in the domain of $B_0$.
\end{definition}

Let $M$ be a merge tree on $n$ nodes and $B$ a homological sequence of length $n$. Using the sublevel-connected Morse order $\leq$ on the nodes of $M,$ we construct the \textbf{homological order consistent with $B$} $\leq_B$ inductively over $B(j)$ as follows: Let $\ell, i,$ and $s$ denote a leaf, inner node, and either a leaf or inner node of $M$, respectively.  Write $\ell_0=s_0\leq s_1\leq \cdots \leq s_n=i_n$ for the sublevel-connected Morse order, and let $b_j$ denote the index of $s$ under the homological order consistent with $B$.
\begin{itemize}
    \item For $j=0$, define $s_{b_0}=\ell_0\leq_B s$ for all nodes $s\in M$.
    \item For $j>0$:
    \begin{enumerate}
        \item If $B(j)>B(j-1)$, then choose the smallest leaf node $\ell_k$ not already chosen and set $s_{b_r}\leq_B \ell_{b_j}=\ell_k$ for all $r\leq j.$
        \item If $B(j)<B(j-1)$, then choose the smallest inner node $\i_k$ not already chosen and set $s_{b_r}\leq_B i_{b_j}=\ell_k$ for all $r\leq j.$
    \end{enumerate}
\end{itemize}
It is clear that $\leq_B$ is a total ordering on the nodes of $M$.

\begin{example}
Suppose we are given a merge tree satisfying
$$
\ell_0\leq \ell_1\leq \ell_2\leq i_3\leq i_4\leq  \ell_5\leq \ell_6\leq i_7\leq \ell_8\leq i_9\leq i_{10}.
$$
We will compute two homological orders from $\leq$ based on two different homological sequences.  Consider the two homological sequences we call $A$ and $B$:

$$
\begin{array}{cccccccccccc}
A: 1 & 2 & 3 & 4 & 3 & 4 & 3& 2 & 1& 2 &1\\
B: 1 & 2 & 1 & 2 & 3 & 2 & 1& 2 & 3&2 & 1.\\
\end{array}
$$
The corresponding homological orders are then seen to be
$$
\ell_0\leq_A \ell_1 \leq_A \ell_2 \leq_A \ell_5 \leq_A i_3 \leq_A \ell_6 \leq_A i_4 \leq_A i_7 \leq_A i_9 \leq_A \ell_8 \leq_A i_{10}
$$
and
$$
\ell_0\leq_B \ell_1 \leq_B i_3 \leq_B \ell_2 \leq_B \ell_5 \leq_B i_4 \leq_B i_7 \leq_B\ell_6 \leq_B \ell_8 \leq_B i_9 \leq_B i_{10}.
$$
\end{example}

We now give our main result.H

\begin{theorem}\label{thm: main} Let $M$ be any chiral merge tree on $n=2k+1>0$ nodes.  If $B_0=B$ is any homological sequence of length $n$ such that $J_0^M\leq B$, then there is a graph $G$ and discrete Morse function $f\colon G\to \NN$ such that $M_f=M$ and $B^f=B$.
\end{theorem}

\begin{proof} Let $M$ be a chiral merge tree on $n=2k+1>0$ nodes and $B$ a homological sequence of length $n$ satisfying $J_0^M\leq B$. Let $\leq_B$ be the homological order consistent with $B$ on the nodes of $M$. We construct a graph $G$ and discrete Morse function $f\colon G\to \NN$ inductively by inducing on the given homological sequence $B_0(j)$ for $0\leq j\leq n-1$.

For the base case $j=0$, we have $B(0)=1$.  Let $G_0=\{v_0\}$ be a single vertex and define $f(v_0)=0$. Now suppose $0<j\leq n-1$. We consider two cases.

If $B(j)>B(j-1)$, then by construction of $\leq_B$, the simplex with index $j$ is a leaf node $\ell_j$.Let $v_j$ be a vertex not in $G_{j-1}$ and define $G_j=G_{j-1}\cup \{v_j\}$ and $f(v_{j})=j.$

If $B(j)<B(j-1)$, then by construction of $\leq_B$, the simplex with index $j$ is an inner node $i_j$. Let $x$ and $y$ be the two children of $i_j$.  Since $x,y\leq_B i_j$,  the nodes $x$ and $y$ correspond to $s_x$ and $s_y$ (either a vertex or an edge) in the graph $G_{i-1}$ as having been constructed in a previous step. Define $G_i=G_{i-1}\cup\{e_j\}$ where $e_j$ is an edge from any vertex in $G_{i-1}[s_x]$ to any vertex in $G_{i-1}[s_y]$ and $f(i_j)=j$.

Let $G=G_{n-1}$.  We claim that $f\colon G\to \NN$ constructed above is the desired discrete Morse function. Observe that by construction, $B_f=B$.

\end{proof}

Given Theorem \ref{thm: main}, we note that it is easy to adjust Algorithm \ref{Jalg} to account for any homological sequence.

\begin{algorithm}
  \caption{Merge tree and homological sequence algorithm}\label{MHalg}
\verb"Input:" A chiral merge tree $M$ on $n=2k+1>0$ nodes and a homological sequence $B$ compatible with $M$ satisfying $B\geq J_0^M$. \\

\verb"Output:" A discrete Morse function $f\colon P \to \mathbb{N}$
whose induced merge tree $M_f$ is equal to $M$ and whose homological sequence $B^f$ is equal to $B$.

\begin{enumerate}
\item[a)] Order the  nodes $s_0\leq_B s_1\leq_B \cdots \leq_B s_{n-1}$ of $M$ using the  homological order consistent with $B, \leq_B.$
\item[b)] Initialize $j=0$, and $P=\{v_0\}$. Set $j=1$.
\item[c)] If $j=n$, end the algorithm. Otherwise
    \begin{enumerate}
     \item   If $s_j$ is a leaf, define $P=P\cup\{v_j\}$ and define $f(v_j)=j$. Set $j=j+1$ and return to c).
    \item    If $s_j$ is an inner node with children $x$ and $y$, define $P=P\cup\{e_j\}$ where $e_j$ is an edge between a vertex of degree 1 in $P[x]$ and a vertex of degree 1 in $P[y]$. Define $f(e_j)=j$.  Set $j=j+1$ and return to c).
     \end{enumerate}
\end{enumerate}
\end{algorithm}

\begin{example}\label{ex: final}
We give an example of how to construct a discrete Morse function with a given homological sequence and merge tree, following Theorem \ref{thm: main}.

Suppose we are given the homological sequence $B$

$$
\begin{array}{ccccccccccc}
1 & 2 & 3 &2 &3 &4 &3 &2 &3 &2 &1 \\
\end{array}
$$
along with merge tree $M$

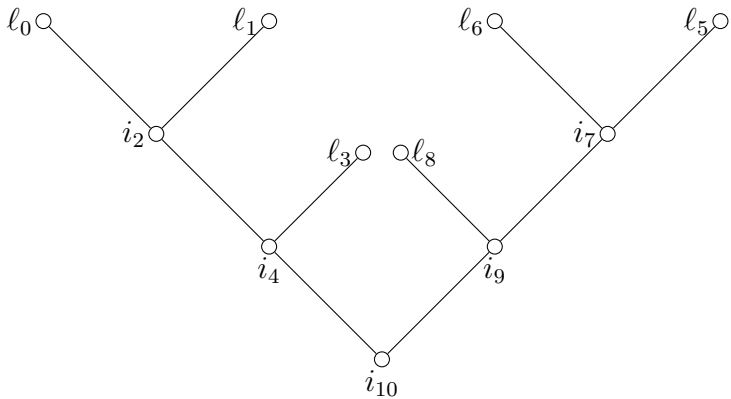
\begin{tikzpicture}

\node[inner sep=2pt, circle] (c) at (1,2.5) [draw] {};
\node[inner sep=2pt, circle] (b) at (-0.5,4) [draw] {};
\node[inner sep=2pt, circle] (p) at (2.5,4) [draw] {};
\node[inner sep=2pt, circle] (h) at (3.75,2.25) [draw] {};
\node[inner sep=2pt, circle] (y) at (5.5,4) [draw] {};
\node[inner sep=2pt, circle] (t) at (5.5,1) [draw] {};
\node[inner sep=2pt, circle] (d) at (4.25,2.25) [draw] {};
\node[inner sep=2pt, circle] (s) at (2.5,1) [draw] {};
\node[inner sep=2pt, circle] (r) at (4,-.5) [draw] {};
\node[inner sep=2pt, circle] (z) at (7,2.5) [draw] {};
\node[inner sep=2pt, circle] (a) at (8.5,4) [draw] {};

\draw[-]  (c)--(s) node[midway, below] {};
\draw[-]  (s)--(h) node[midway, below] {};
\draw[-]  (s)--(r) node[midway, below] {};
\draw[-]  (t)--(r) node[midway, below] {};
\draw[-]  (t)--(z) node[midway, below] {};
\draw[-]  (t)--(d) node[midway, below] {};
\draw[-]  (z)--(y) node[midway, below] {};
\draw[-]  (c)--(b) node[midway, below] {};
\draw[-]  (c)--(p) node[midway, below] {};
\draw[-]  (z)--(a) node[midway, below] {};

\node[anchor = east]  at (b) {{$\ell_0$}};
\node[anchor = east]  at (c) {{$i_2$}};
\node[anchor = east]  at (h) {{$\ell_3$}};
\node[anchor = north]  at (s) {{$i_4$}};
\node[anchor = north]  at (r) {{$i_{10}$}};
\node[anchor = north]  at (t) {{$i_9$}};
\node[anchor = east]  at (p) {{$\ell_1$}};
\node[anchor = east]  at (y) {{$\ell_6$}};
\node[anchor = west]  at (d) {{$\ell_8$}};
\node[anchor = east]  at (z) {{$i_7$}};
\node[anchor = east]  at (a) {{$\ell_5$}};
\end{tikzpicture}

The above labelings of the nodes of $M$ are chosen with the homological sequence $B$ in mind.  Then the homological order consistent with $B$ is then seen to be
$$
\ell_0 \leq \ell_1 \leq i_2 \leq \ell_3 \leq i_4 \leq \ell_5\leq \ell_6 \leq i_7 \leq \ell_8 \leq i_9 \leq i_{10}
$$
where $\leq=\leq_B$.

We will construct a discrete Morse function $f\colon P_5\to \NN$ on the path $P_5$. The first step is to construct two vertices $v_0$ and $v_1$ corresponding to nodes $\ell_0$ and $\ell_1$, respectively. Define $f(v_0)=0$ and $f(v_1)=1$.

The given homological sequence satisfies $B(2)>B(1)$, so we add another vertex $v_2$ corresponding to $\ell_2$ and define $f(v_2)=2$.  Since $B(3)<B(2)$, we need to add an edge $e_3$ to the graph.  This edge corresponds to the first node in the inner node ordering $i_3$, and it will have endpoints $v_0$ and $v_1$. Define $f(e_3)=3$. Next we have $B(4)>B(3)$ so add a vertex $v_4$ corresponding to $\ell_5$ and define $f(v_4)=4$. Next $B(5)>B(4)$ so we add a vertex $v_5$ corresponding to the next node in the ordering $\leq_{\ell}$, namely, $\ell_5$.
We so far have
$$
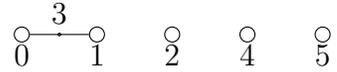
\begin{tikzpicture}

\node[inner sep=2pt, circle] (a) at (-2,0) [draw] {};
\node[inner sep=2pt, circle] (b) at (-1,0) [draw] {};
\node[inner sep=2pt, circle] (c) at (0,0) [draw] {};
\node[inner sep=2pt, circle] (d) at (1,0) [draw] {};
\node[inner sep=2pt, circle] (e) at (2,0) [draw] {};
\node[inner sep=0pt, circle] (f) at (-1.5,0) [draw] {};

\draw[-]  (a)--(b) node[midway, below] {};

\node[anchor = north]  at (a) {{$0$}};
\node[anchor = north]  at (b) {{$1$}};
\node[anchor = north]  at (c) {{$2$}};
\node[anchor = north]  at (d) {{$4$}};
\node[anchor = north]  at (e) {{$5$}};
\node[anchor = south]  at (f) {{$3$}};

\end{tikzpicture}
$$
as our partial path.

We define $f(v_5)=5$. Now $B(6)<B(5)$ so we pick the next inner node $i_4$ in the ordering $\leq_i$ and construct edge $e_6$ with endpoints $v_1$ and $v_2$ along with $f(e_6)=6$. Since $B(7)<B(6)$, we pick inner node $i_7$ and construct edge $e_7$ with endpoints $v_4$ and $v_5$ and $f(e_7)=7$. Now $B(8)>B(7)$ so we add a vertex $v_8$ corresponding to node $\ell_8$ with label $f(v_8)=8$.

We now have
$$
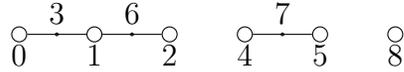
\begin{tikzpicture}

\node[inner sep=2pt, circle] (a) at (-2,0) [draw] {};
\node[inner sep=2pt, circle] (b) at (-1,0) [draw] {};
\node[inner sep=2pt, circle] (c) at (0,0) [draw] {};
\node[inner sep=2pt, circle] (d) at (1,0) [draw] {};
\node[inner sep=2pt, circle] (e) at (2,0) [draw] {};
\node[inner sep=0pt, circle] (f) at (-1.5,0) [draw] {};
\node[inner sep=2pt, circle] (g) at (3,0) [draw] {};
\node[inner sep=0pt, circle] (h) at (-0.5,0) [draw] {};
\node[inner sep=0pt, circle] (i) at (1.5,0) [draw] {};

\draw[-]  (a)--(b) node[midway, below] {};
\draw[-]  (b)--(c) node[midway, below] {};
\draw[-]  (d)--(e) node[midway, below] {};

\node[anchor = north]  at (a) {{$0$}};
\node[anchor = north]  at (b) {{$1$}};
\node[anchor = north]  at (c) {{$2$}};
\node[anchor = north]  at (d) {{$4$}};
\node[anchor = north]  at (e) {{$5$}};
\node[anchor = south]  at (f) {{$3$}};
\node[anchor = north]  at (g) {{$8$}};
\node[anchor = south]  at (h) {{$6$}};
\node[anchor = south]  at (i) {{$7$}};

\end{tikzpicture}
$$
as our partial path.

We then have $B(9)<B(8)$ so we add edge $e_9$ corresponding to $i_9$ with endpoints $v_5$ and $v_8$ along with label $f(e_9)=9$.  Finally, $B(10)<B(9)$ and we construct edge $e_{10}$ corresponding to the root node $i_{10}$ with endpoints $v_2$ and $v_4$ with label $f(e_{10})=10.$

By construction, the discrete Morse function $f$ on this graph has the property that $M_f=M$ and $B_f=B.$
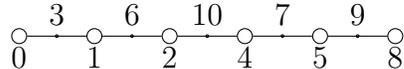
$$
\begin{tikzpicture}

\node[inner sep=2pt, circle] (a) at (-2,0) [draw] {};
\node[inner sep=2pt, circle] (b) at (-1,0) [draw] {};
\node[inner sep=2pt, circle] (c) at (0,0) [draw] {};
\node[inner sep=2pt, circle] (d) at (1,0) [draw] {};
\node[inner sep=2pt, circle] (e) at (2,0) [draw] {};
\node[inner sep=0pt, circle] (f) at (-1.5,0) [draw] {};
\node[inner sep=2pt, circle] (g) at (3,0) [draw] {};
\node[inner sep=0pt, circle] (h) at (-0.5,0) [draw] {};
\node[inner sep=0pt, circle] (i) at (1.5,0) [draw] {};
\node[inner sep=0pt, circle] (j) at (2.5,0) [draw] {};
\node[inner sep=0pt, circle] (k) at (0.5,0) [draw] {};

\draw[-]  (a)--(b) node[midway, below] {};
\draw[-]  (b)--(c) node[midway, below] {};
\draw[-]  (d)--(e) node[midway, below] {};
\draw[-]  (e)--(g) node[midway, below] {};
\draw[-]  (c)--(d) node[midway, below] {};

\node[anchor = north]  at (a) {{$0$}};
\node[anchor = north]  at (b) {{$1$}};
\node[anchor = north]  at (c) {{$2$}};
\node[anchor = north]  at (d) {{$4$}};
\node[anchor = north]  at (e) {{$5$}};
\node[anchor = south]  at (f) {{$3$}};
\node[anchor = north]  at (g) {{$8$}};
\node[anchor = south]  at (h) {{$6$}};
\node[anchor = south]  at (i) {{$7$}};
\node[anchor = south]  at (j) {{$9$}};
\node[anchor = south]  at (k) {{$10$}};

\end{tikzpicture}
$$
\end{example}

\begin{example}\label{ex: not all compatible} It is not the case that any homological sequence of length $m$ is compatible with any tree on $m$ nodes.  Indeed, consider the following merge tree

$$
\begin{tikzpicture}

\node[inner sep=2pt, circle] (0) at (0,0) [draw] {};
\node[inner sep=2pt, circle] (1) at (-1,1) [draw] {};
\node[inner sep=2pt, circle] (2) at (1,1) [draw] {};
\node[inner sep=2pt, circle] (3) at (2,2) [draw] {};
\node[inner sep=2pt, circle] (4) at (0.25,1.75) [draw] {};
\node[inner sep=2pt, circle] (5) at (-2,2) [draw] {};
\node[inner sep=2pt, circle] (6) at (-.25,1.75) [draw] {};

\draw[-]  (0)--(1) node[midway, below] {};
\draw[-]  (0)--(2) node[midway, below] {};
\draw[-]  (2)--(3) node[midway, below] {};
\draw[-]  (2)--(4) node[midway, below] {};
\draw[-]  (1)--(5) node[midway, below] {};
\draw[-]  (1)--(6) node[midway, below] {};


\end{tikzpicture}
$$
along with homological sequence

$$
\begin{array}{ccccccc}
1 & 2 & 1 &2 &1 &2 &1  \\
\end{array}
$$

$$
\begin{tikzpicture}

\node[inner sep=2pt, circle] (0) at (0,0) [draw] {};
\node[inner sep=2pt, circle] (1) at (-1,1) [draw] {};
\node[inner sep=2pt, circle] (2) at (1,1) [draw] {};
\node[inner sep=2pt, circle] (3) at (2,2) [draw] {};
\node[inner sep=2pt, circle] (4) at (0.25,1.75) [draw] {};
\node[inner sep=2pt, circle] (5) at (-2,2) [draw] {};
\node[inner sep=2pt, circle] (6) at (-.25,1.75) [draw] {};

\draw[-]  (0)--(1) node[midway, below] {};
\draw[-]  (0)--(2) node[midway, below] {};
\draw[-]  (2)--(3) node[midway, below] {};
\draw[-]  (2)--(4) node[midway, below] {};
\draw[-]  (1)--(5) node[midway, below] {};
\draw[-]  (1)--(6) node[midway, below] {};

\node[anchor = north]  at (0) {{$6$}};
\node[anchor = north]  at (1) {{$2$}};
\node[anchor = north]  at (2) {{$5$}};
\node[anchor = north]  at (3) {{$3$}};
\node[anchor = north]  at (4) {{$4$}};
\node[anchor = north]  at (5) {{$0$}};
\node[anchor = north]  at (6) {{$1$}};


\end{tikzpicture}
$$
\end{example}

\subsection{Open questions}

We share a few open questions that our work has raised.

\begin{enumerate}
    \item Example \ref{ex: final} raises the question of whether or not any sequence $B<J_0$ that is consistent with a given merge tree $M$ can be realized by a discrete Morse function. We conjecture in the negative.
    \item In \cite{RandScoville}, the authors show that given a gradient vector field on a tree and any homological sequence with the same number of critical simplices as the number of critical simplices on the gradient vector field, there is a discrete Morse function on the tree $T$ that induces both the given gradient vector field and homological sequence.  One may ask a similar question about the compatibility between a given gradient vector field and a given merge tree $M$.  This question seems to be much more delicate than the one taken up in this paper since the structure of the given tree $T$ greatly affects which merge trees can be induced.
    \item Relating to the previous point, given any fixed tree or graph $G$ in general, can one compute the collection of all merge trees that can be induced by a discrete Morse function on $G$?  We know that any merge tree may be realized as the image of some discrete Morse function on some path, but what does the set of all realizable merge trees for a fixed graph look like?
\end{enumerate}

\providecommand{\bysame}{\leavevmode\hbox to3em{\hrulefill}\thinspace}
\providecommand{\MR}{\relax\ifhmode\unskip\space\fi MR }
\providecommand{\MRhref}[2]{%
  \href{http://www.ams.org/mathscinet-getitem?mr=#1}{#2}
}
\providecommand{\href}[2]{#2}

\end{document}